\documentclass[11pt,a4paper]{article}
\usepackage[T1]{fontenc}
\usepackage[utf8]{inputenc}
\usepackage[american]{babel}
\usepackage{lmodern}
\usepackage[left=2.2cm,right=2.2cm,top=2.1cm,bottom=2.1cm]{geometry}
\usepackage{mathtools,amssymb,amsthm,bm}
\usepackage{microtype}
\usepackage{tikz}
\usetikzlibrary{decorations.pathreplacing}
\usepackage[font=small,labelfont=bf]{caption}
\usepackage{flafter}
\usepackage[colorlinks=true,anchorcolor=blue,filecolor=blue,
  linkcolor=red,urlcolor=blue,citecolor=blue]{hyperref}
  
\AtBeginDocument{%
  \setlength{\abovedisplayskip}{4pt plus 2pt minus 3pt}%
  \setlength{\belowdisplayskip}{4pt plus 2pt minus 3pt}%
  \setlength{\abovedisplayshortskip}{3pt plus 2pt minus 1pt}%
  \setlength{\belowdisplayshortskip}{3pt plus 2pt minus 2pt}%
}

\newtheorem{theorem}{Theorem}[section]
\newtheorem{proposition}[theorem]{Proposition}

\renewcommand{\le}{\leqslant}

\renewcommand{\ge}{\geqslant}

\title{The maximum degree of Nosal graphs}

\author{Yongtao Li\thanks{Yau Mathematical Sciences Center, Tsinghua University, Beijing, China. E-mail address: \texttt{ytli0921@hnu.edu.cn}}
 }
 
\date{}

\begin{document}
\maketitle

\vspace{-9mm}
\begin{abstract}
A graph $G$ with $m$ edges is a Nosal graph if its spectral radius $\lambda(G)$ exceeds $\sqrt m$.
At the threshold $\lambda(G)=\sqrt m$, the star $K_{1,m}$ shows that one vertex can be incident with all $m$ edges.
In contrast, Li, Liu and Zhang [J. Combin. Theory Ser. B 179 (2026)] proved that every Nosal graph $G$ with sufficiently many edges satisfies $\Delta(G)\le\frac12m+m^{0.99}$, where the constant $\frac12$ is best possible.
In this paper, we prove that every Nosal graph $G$ with spectral radius $\lambda$ satisfies $\Delta(G)<m-\frac12(\lambda-1)(\lambda-2)$, and hence $\Delta(G)<\frac12m+\frac32\sqrt m-1$.
Both bounds are sharp up to an additive constant less than $\frac12$: for every $r\ge3$, the complete graph $K_r$ with $\frac12(r^2+r-2)$ pendant vertices attached to one vertex is a Nosal graph with $r^2-1$ edges whose maximum degree lies within $\frac12$ of both bounds.
Consequently, the largest maximum degree of a Nosal graph with $r^2-1$ edges is exactly $\frac12(r^2+3r-4)$.
The first bound cannot be improved by any positive additive constant: 
its gap on these graphs is less than $\frac{4}{r+5}$ and tends to zero as $r\to \infty$.
\end{abstract}

\textbf{Keywords:} Nosal graph; spectral radius; maximum degree; Perron vector.

\textbf{2020 Mathematics Subject Classification:} 05C50; 05C35.

\section{Introduction}\label{sec-1}

For a graph $G$, we write $n=|V(G)|$ and $m=e(G)$ for its numbers of vertices and edges, $d_G(v)$ for the degree of a vertex $v$, and $\Delta(G)$ for its maximum degree.
Let $A(G)$ be the adjacency matrix of $G$, and let $\lambda(G)$ be its largest eigenvalue, which is also its spectral radius.
Spectral extremal graph theory asks how large the spectral radius of a graph can be when the graph contains no copy of a given graph $F$.
In the edge-spectral setting, which goes back to Brualdi and Hoffman~\cite{BH1985}, the number of edges is prescribed.
The starting point of this setting is Nosal's theorem~\cite{Nosal1970}: every triangle-free graph $G$ satisfies $\lambda(G)\le\sqrt m$.
Several refinements and generalizations of Nosal's theorem are known, we refer the readers to \cite{BN2007,LinNingWu2021,Nikiforov2021,NingZhai2023,ZhaiLinShu2021}.
Nikiforov~\cite{Nikiforov2002} extended it to cliques by proving that every $K_{r+1}$-free graph $G$ satisfies $\lambda^2(G)\le (1-1/r)2m$; see, e.g., \cite{LiLiuZhang2025a,LiLiuZhang2025b} for further extensions.

\smallskip
Since $\lambda(G)\ge2m/n$, Nosal's theorem implies Mantel's theorem: every triangle-free graph 
on $n$ vertices contains at most $n^2/4$ edges.
The spectral condition $\lambda (G) > \sqrt{m}$, however, involves only the number of edges, 
so it also applies to sparse graphs with any edge density.
For example, let the \emph{book} $B_k = K_2 \vee \overline{K}_k$ consist of $k$ triangles that share one common edge.
For odd $m\ge3$, the book $B_{(m-1)/2}$ has $m$ edges but only $\frac{1}{2}(m+3)$ vertices, and its spectral radius $\frac{1}{2}(1+\sqrt{4m-3})$ exceeds $\sqrt m$.
The threshold $\sqrt m$ is sharp: every complete bipartite graph $K_{a,b}$ with $ab=m$ is triangle-free and satisfies $\lambda(K_{a,b})=\sqrt m$.
Following Li, Liu and Zhang~\cite{LiLiuZhang2026}, we call $G$ a {\color{blue}\bf \emph{Nosal graph}} if $\lambda(G)>\sqrt m$.
Since isolated vertices change neither $m$ nor $\lambda(G)$, we allow them throughout, and structural statements such as ``$G$ is complete bipartite'' are understood up to isolated vertices.

\subsection{Substructures in Nosal graphs}

In fact, a Nosal graph contains many triangles.
Let $t(G)$ be the number of triangles in $G$, and write $\lambda=\lambda(G)$.
Bollob\'as and Nikiforov~\cite{BN2007} proved that $t(G)\ge \frac{1}{3} \lambda(\lambda^2-m)$.
Li, Liu and Ning~\cite{LiLiuNing2026} recently improved this to $t(G)\ge \frac{1}{3} (\lambda+1)(\lambda^2-m)$, with equality exactly for complete graphs and complete bipartite graphs.
Chen, Li and Tang~\cite{ChenLiTang2026} proved the linear bound $t(G)\ge m(\lambda-\sqrt m)$.
Ning and Zhai~\cite{NingZhai2023} showed that $\lambda\ge\sqrt m$ forces at least $\lfloor \frac{1}{2}(\sqrt m-1)\rfloor$ triangles unless $G$ is complete bipartite.
Further above the threshold, the guaranteed number of triangles jumps from order $\sqrt m$ to order $m$.
For sufficiently large $m$, Zhang and Zhai~\cite{ZhangZhai2026} proved that $\lambda\ge \frac{1}{2}(1+\sqrt{4m-3})$ forces $t(G)\ge \frac{1}{2}(m-1)$. Chen, Li and Tang~\cite{ChenLiTang2026} proved that $\lambda\ge1+\sqrt{m-2}$ forces $t(G)\ge m-2$.
Both bounds are attained by split graphs: the book $B_{(m-1)/2}$ for odd $m$, and, for $m$ divisible by $3$, the graph obtained from a triangle by adding $(m-3)/3$ new vertices adjacent to all three of its vertices.

\smallskip
Supersaturation phenomena in Nosal graphs also occur for other subgraphs, and Table~\ref{tab-1} lists the current bounds.
Nikiforov~\cite{Nikiforov2009} proved that every Nosal graph with $m\ge 10$ edges contains a four-cycle $C_4$.
Ning and Zhai~\cite{NingZhai2025} showed that every Nosal graph with sufficiently many edges contains at least $\frac{1}{2000}m^2$ copies of $C_4$, which is the correct order of magnitude.
Li, Liu and Zhang~\cite{LiLiuZhang2026} improved the constant to $\frac{1}{8}-o(1)$, which is best possible.
Zhai, Lin and Shu~\cite{ZhaiLinShu2021} proved that every Nosal graph contains a complete bipartite graph $K_{2,t}$ with $t>\frac{1}{4} \sqrt m$.
Li, Liu and Zhang~\cite{LiLiuZhang2025a} improved this to $t\ge \frac{1}{2}\sqrt m-O(1)$, and random graphs show that the constant $\frac{1}{2}$ is best possible.

\smallskip
Large books in Nosal graphs have received particular attention in recent years.
Nikiforov~\cite{Nikiforov2021} proved that every Nosal graph contains a book $B_k$ with size $k> \frac{1}{12}m^{1/4}$, settling a conjecture of Zhai, Lin and Shu~\cite[Conjecture~5.2]{ZhaiLinShu2021}.
Using a probabilistic argument, Li, Liu and Zhang~\cite{LiLiuZhang2026} obtained the correct order, namely $k>\frac{1}{24} \sqrt m$.
They also constructed Nosal graphs showing that the constant cannot exceed $\frac13$, and they asked whether $k>\frac{1}{3} \sqrt m$ always holds.
The constant was improved to $\frac{1}{9}$ by Zhai, Li and Lou~\cite{ZhaiLiLou2026} and to $\frac{1}{4}$ by Chen, Li and Tang~\cite{ChenLiTang2026}.
Zhao, You, Zeng and Zhang~\cite{ZYZZ2026} answered this question affirmatively in the stronger form $k\ge\lambda/3$.
Li, Liu and Ning~\cite{LiLiuNing2026} then proved $k\ge\lambda-2m/(3\lambda)$, which is tight for infinitely many graphs.
The book $B_2$ is also called the \emph{kite} and denoted by $C_4^+$, since it arises from $C_4$ by adding a chord.
Chen, Li and Tang~\cite{ChenLiTang2026} showed that every Nosal graph contains at least $(\frac{1}{8} -o(1))m$ kites, and the constant $\frac{1}{8}$ is best possible.

\begin{table}[htbp]
\centering
\small
\renewcommand{\arraystretch}{1.25}
\begin{tabular}{llll}
\hline
Structure & Bound & Sharpness & References\\
\hline
$\# C_3$ & $\ge \frac{1}{3}(\lambda+1)(\lambda^2-m)$ & tight for $K_s$ and $K_{a,b}$ & \cite{BN2007,LiLiuNing2026}\\
$\# C_3$ & $\ge m(\lambda-\sqrt m\,)$ & tight for $K_{a,b}$ & \cite{ChenLiTang2026}\\
$\# C_4$ & $\ge ( \frac18-o(1))m^2$ & constant $1/8$ & \cite{NingZhai2025,LiLiuZhang2026}\\
$\# C_4^+$ & $\ge ( \frac18-o(1))m$ & constant $1/8$ & \cite{ChenLiTang2026}\\
Large $K_{2,t}$ & $t\ge \frac{1}{2} \sqrt m-O(1)$ & constant $1/2$ & \cite{ZhaiLinShu2021,LiLiuZhang2025a}\\
Large $B_k$ & $k> \frac{1}{3}\sqrt{m}$ & constant $1/3$ &
\cite{ZYZZ2026,LiLiuNing2026}\\
\hline
\textbf{maximum degree} & $\bm{\Delta(G)<\frac{m}{2} \!+\! \frac{3}{2} \sqrt{m} \! -\!1}$ & \textbf{coefficient $\bm{3/2}$} & \textbf{Theorem~\ref{thm-1.1}}\\
\hline
\end{tabular}
\caption{Extremal results for a Nosal graph $G$ with $m$ edges and spectral radius $\lambda=\lambda(G)$.
The upper rows give substructures that $G$ must contain; the two triangle bounds hold for every graph, and $o(1)\to0$ as $m\to\infty$.
The last row is the main result of this paper: no vertex of $G$ has degree $m/2+(3/2)\sqrt m-1$ or more (Theorem~\ref{thm-1.1}), and for $m=r^2-1$ the largest integer below this bound is attained (Proposition~\ref{prop-1.3}).}
\label{tab-1}
\end{table}

Near the threshold $\lambda=\sqrt m$, graphs without large books are close to complete bipartite graphs.
Li, Liu and Zhang~\cite[Theorem~4.1]{LiLiuZhang2025a} proved a stability version of Nosal's theorem: for every $\varepsilon\in(0,0.01)$, there is $\delta = \delta (\varepsilon)>0$ such that every triangle-free graph with $\lambda(G)\ge(1-\delta)\sqrt m$ becomes complete bipartite after adding and deleting at most $\varepsilon m$ edges.
Li, Liu and Ning~\cite[Theorem 1.9]{LiLiuNing2026} proved a stronger statement for all graphs: for every $\varepsilon>0$, there is $\delta = \delta (\varepsilon )>0$ such that every graph with $\lambda(G)\ge(1-\delta)\sqrt m$ either contains a book $B_k$ with size $k\ge( \frac 13-\varepsilon)\sqrt m$, or becomes complete bipartite after adding and deleting at most $\varepsilon m$ edges.

\smallskip
For all Nosal graphs, Li, Liu and Zhang~\cite[Theorem~1.8]{LiLiuZhang2026} proved a structure theorem: for every $\varepsilon>0$, there is a constant $N(\varepsilon)$ such that every Nosal graph $G$ with $m$ edges has a subgraph $G'$ with $\lambda(G')>(1-\varepsilon)\sqrt m-N(\varepsilon)$ that is bipartite or has maximum degree at most $\varepsilon m$.
The book $B_{(m-1)/2}$ and complete graphs show that both cases are needed.
The bipartite case does not exclude a vertex of degree close to $m$: the star $K_{1,m}$ is bipartite and has spectral radius $\sqrt m$.
This raises the question of how unevenly the edges of a Nosal graph can be distributed among its vertices.

\subsection{The maximum degree}

Vertex degrees also enter the counting problems above.
For example, let $\# C_4(G)$ be the number of subgraphs of $G$ isomorphic to $C_4$.
A standard count of closed walks of length four~\cite{LL2009} gives
\[
8\# C_4(G)=\operatorname{tr}A(G)^4-2\sum_{v\in V(G)}d_G(v)^2+2m.
\]
The degree terms count the closed walks of length four that do not trace a four-cycle, and a vertex of large degree lies on many such walks.
Therefore, turning spectral information into a count of four-cycles requires control of the high-degree vertices in a Nosal graph; see, for example,~\cite{NingZhai2025,LiLiuZhang2026}.

\smallskip 
As a key step in their count of four-cycles, Li, Liu and Zhang~\cite{LiLiuZhang2026} showed that every Nosal graph $G$ with sufficiently many edges satisfies
\begin{equation}\label{e-1}
\Delta(G)\le\frac m2+m^{0.99}.
\end{equation}
Using this bound, they showed~\cite[Corollary~4.5]{LiLiuZhang2026} that every Nosal graph $G$ with $m$ edges satisfies
\[
\sum_{v\in V(G)}d_G (v)^2 \le \Bigl(\frac{1}{2} + o(1) \Bigr)m^2.
\]
These two bounds play a key role in their proof of the sharp bound $\#C_4 (G)\ge (\frac{1}{8}- o(1))m^2$.
Li, Liu and Zhang~\cite{LiLiuZhang2025a} also used~\eqref{e-1} to find $K_{2,t}$ with $t\ge \frac{1}{2}\sqrt{m} - O(1)$.
The leading term $m/2$ in~\eqref{e-1} is best possible: for odd $m\ge3$, the book $B_{(m-1)/2}$ is a Nosal graph, and the two ends of its common edge have degree $\frac{1}{2}(m+1)$.
This bound reveals another jump at the threshold: the star $K_{1,m}$ has all $m$ edges at one vertex, whereas a Nosal graph has at most about half of its edges at any vertex.

\smallskip
Li, Liu and Zhang~\cite{LiLiuZhang2026} also noted that the error term $m^{0.99}$ in~\eqref{e-1} is far from optimal but suffices for their purpose.
It is natural to ask how large the correction to $m/2$ can be.
We show that it is less than $\frac32\sqrt m-1$, and this bound is sharp up to an additive error below $\frac12$ for infinitely many $m$.

\begin{theorem}\label{thm-1.1}
Every Nosal graph $G$ with $m$ edges satisfies
\[
\Delta(G)\ < \ \frac m2+\frac32\sqrt m-1.
\]
\end{theorem}

Theorem~\ref{thm-1.1} improves the error term in~\eqref{e-1} to the explicit term $\frac{3}{2} \sqrt m -1$. We note that Theorem~\ref{thm-1.1} holds for every Nosal graph, not only for those with sufficiently many edges.

\begin{theorem}\label{thm-1.2}
Every Nosal graph $G$ with $m$ edges satisfies
\[
\Delta(G)<m-\frac{1}{2} (\lambda(G) -1)(\lambda(G) -2).
\]
\end{theorem}

This bound has a simple meaning.
Since $e(G-v)=m-d_G(v)$, Theorem~\ref{thm-1.2} says that 
\[
e(G-v)>\frac{1}{2} (\lambda(G)-1)(\lambda(G)-2) \quad\text{for every }v\in V(G).
\]
Thus, a larger spectral radius forces more edges to avoid each vertex.

\smallskip 
Denote $f(s):= \frac 12 (s-1)(s-2)$, so that Theorem~\ref{thm-1.2} reads $\Delta(G)<m-f(\lambda(G))$. 
Note that every Nosal graph $G$ has $m\ge3$ edges, and $\lambda(G)>\sqrt m\ge\sqrt3 $.
Since $f(s)$ is strictly increasing on $s\in [3/2,\infty)$, Theorem~\ref{thm-1.2} gives
$\Delta(G)<m-f(\sqrt m)=\frac m2+\frac32\sqrt m-1$, 
which is Theorem~\ref{thm-1.1}.

\smallskip 
Both bounds in Theorems \ref{thm-1.1} and \ref{thm-1.2} are sharp up to an additive constant less than $\frac12$.
For $r\ge3$ and $t\ge0$, let $K_r^t$ be the graph obtained from $K_r$ by attaching $t$ pendant vertices to a single vertex of $K_r$; see Figure~\ref{fig-1}.
Such graphs are known as {\it pineapple} graphs; see, e.g., Tait and Tobin \cite{TT2017}.

\begin{figure}[htbp]
\centering
\begin{tikzpicture}[x=1cm,y=1cm,
  vtx/.style={circle,fill=black,inner sep=1.6pt}]
  \foreach \i in {0,...,4}
    \coordinate (c\i) at ({180+72*\i}:1.15);
  \foreach \i in {0,...,3}
    \foreach \j in {\i,...,4}
      { \draw (c\i)--(c\j); }
  \foreach \i in {1,...,4} \node[vtx] at (c\i) {};
  \node[vtx,label={[label distance=1pt]below:{$v$}}] (v) at (c0) {};
  \foreach \y in {1.35,0.75,-0.75,-1.35}
    { \node[vtx] (p) at (-3.3,\y) {}; \draw (v)--(p); }
  \node at (-3.3,0.05) {$\vdots$};
  \draw[decorate,decoration={brace,amplitude=5pt}]
    (-3.6,-1.45)--(-3.6,1.45)
    node[midway,left=7pt,align=right]{$t$ pendant\\vertices};
  \node at (1.6,0) {$K_r$};
\end{tikzpicture}
\caption{The graph $K_r^t$, drawn for $r=5$: a complete graph $K_r$ with $t$ pendant vertices attached to the vertex $v$.
For $t=(r^2+r-2)/2$, it is a Nosal graph with $m=r^2-1$ edges, and the degree $(r^2+3r-4)/2$ of $v$ lies within $\frac12$ of the bound in Theorem~\ref{thm-1.1} (Proposition~\ref{prop-1.3}).}
\label{fig-1}
\end{figure}
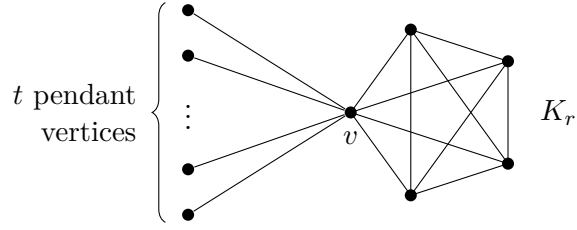

\begin{proposition}\label{prop-1.3}
Let $r\ge3$ be an integer, let $t=(r^2+r-2)/2$, and let $G=K_r^t$.
Then $G$ has $m=r^2-1$ edges, $\Delta(G)=(r^2+3r-4)/2$, and $\sqrt m<\lambda(G)<r$.
Thus $G$ is a Nosal graph, and
\begin{align*}
0&<\frac m2+\frac32\sqrt m-1-\Delta(G)<\frac12,\\
0&<m-\frac{1}{2} (\lambda(G)-1)(\lambda(G)-2) -\Delta(G)<\frac{4}{r+5}.
\end{align*}
\end{proposition}

Since maximum degrees are integers, Theorem~\ref{thm-1.1} and Proposition~\ref{prop-1.3} determine the extremal value for every $m=r^2-1$: the maximum of $\Delta(G)$ over all Nosal graphs $G$ with $r^2-1$ edges is exactly $\frac 12 (r^2+3r-4)$, which is the largest integer less than $\frac m2+\frac32\sqrt m-1$.
In particular, the coefficient $\frac32$ of the term $\sqrt{m}$ in Theorem~\ref{thm-1.1} cannot be lowered, and the constant term $-1$ cannot be replaced by $-\frac32$.
The second inequality in Proposition~\ref{prop-1.3} shows that the stronger bound in Theorem~\ref{thm-1.2} is also tight up to an additive error less than $4/(r+5)\le\frac12$, which tends to $0$ as $r\to\infty$.
Hence, no positive constant can be subtracted from the right side of Theorem~\ref{thm-1.2}.

\section{Proof of Theorem \ref{thm-1.2}} \label{sec-2}

We fix a vertex $v$ and normalize the eigenvector so that the two largest coordinates outside $v$ sum to $1$.
Two counts based on the eigenvector equations show that more than $f(\lambda(G))$ edges avoid $v$.
 
For a vertex $i$ of a graph $G$, let $N_G(i)$ be the neighborhood of $i$, and let $G-i$ be the graph obtained by deleting $i$ and its incident edges.
For $S\subseteq V(G)$, let $G[S]$ be the subgraph of $G$ induced by $S$.
By the Perron--Frobenius theorem, $A(G)$ has a nonnegative eigenvector $\bm{x}$ for $\lambda(G)$.
If $G$ is connected, this eigenvector can be chosen positive and is unique up to a positive scalar factor; we call it a \emph{Perron vector} of $G$.
If $G$ is disconnected, some coordinates of this vector may be zero.

\begin{proof}[Proof of Theorem~\ref{thm-1.2}]
Let $\lambda=\lambda(G)$.
By Nosal's theorem, $G$ contains a triangle, so $G$ has at least three vertices and $\lambda\ge\lambda(K_3)=2$.
Fix a vertex $v$, and let $q :=e(G-v)=m-d_G(v)$. 
Recall that $f(s) = \frac 12 (s-1)(s-2)$. 
It suffices to show that $q>f(\lambda)$.
Let $\bm{x}$ be a nonnegative eigenvector of $A(G)$ for $\lambda$.
Some vertex other than $v$ has a positive coordinate, since otherwise $\lambda x_v=\sum_{i\in N_G(v)}x_i=0$ and $\bm{x}=\bm{0}$.
Choose distinct vertices $u,w\ne v$ with the two largest coordinates among the vertices other than $v$, and normalize $\bm{x}$ so that $x_u+x_w=1$.
Let $U:=V(G)\setminus\{u,v,w\}$.
By the choice of $u$ and $w$, we have $x_i+x_j\le1$ for all distinct $i,j\ne v$, and $x_i\le\frac12$ for all $i\in U$.

Applying the eigenvector equation twice at $v$, we obtain
\[
\lambda^2 x_v =\sum_{i\in N_G(v)}\lambda x_i=d_G(v) x_v+\sum_{i\in N_G(v)}\ \sum_{j\in N_G(i)\setminus\{v\}}x_j.
\]
An edge $ij$ of $G-v$ contributes $x_j$ to the double sum if $i\in N_G(v)$, and $x_i$ if $j\in N_G(v)$.
Hence it contributes at most $x_i+x_j\le1$, and $(\lambda^2-d_G(v)) x_v \le q$.
Since $\lambda^2>m=d_G(v)+q$, we have $\lambda^2-d_G(v)>q\ge0$, and therefore $x_v <1$.

Let $h=1$ if $uw\in E(G)$ and $h=0$ otherwise.
For $i\in U$, let $c_i :=|N_G(i)\cap\{u,w\}|$.
Counting the edges of $G-v$ by their ends in $\{u,w\}$, we get
\[
q=h+\sum_{i\in U}c_i+e(G[U]).
\]
Adding the eigenvector equations at $u$ and $w$ and using $x_u+x_w=1$, we obtain
\[
\lambda\le2x_v +h+\sum_{i\in U}c_ix_i.
\]
For $i\in U$, the neighbors of $i$ in $\{u,v,w\}$ contribute at most $ x_v +1$ to $\lambda x_i$, and each neighbor of $i$ in $U$ contributes at most $\frac12$.
Thus $\lambda x_i\le x_v +1+\frac12d_{G[U]}(i)$.
Multiplying the inequality obtained at $u$ and $w$ by $\lambda$, applying this bound, and using $c_i\le2$, we obtain
\[
\lambda^2\le(2x_v+h)\lambda+(x_v+1)\sum_{i\in U}c_i+\frac12\sum_{i\in U}c_i\,d_{G[U]}(i)
\le(2x_v+h)\lambda+(x_v +1)\sum_{i\in U}c_i+2e(G[U]).
\]
Substituting $e(G[U])=q-h-\sum_{i\in U}c_i$ and rearranging, we obtain
\[
2q-(\lambda-1)(\lambda-2)\ge(1-x_v)\left(2\lambda+\sum_{i\in U}c_i\right)+(1-h)(\lambda-2).
\]
Since $x_v <1$, $h\le1$ and $\lambda\ge2$, 
the right side is at least $2\lambda(1- x_v)>0$.
Hence $2q>(\lambda-1)(\lambda-2)$, that is, $q>f(\lambda)$.
As $v$ is an arbitrary vertex, we conclude that $\Delta(G)<m-f(\lambda)$.
\end{proof}

\section{Proof of Proposition~\ref{prop-1.3}}\label{sec-3}

The graph $K_r^t$ balances the edges of a clique against pendant edges at one of its vertices.
We first count its edges and degrees and then locate its spectral radius between $\sqrt m$ and $r$.

\begin{proof}[Proof of Proposition~\ref{prop-1.3}]
Let $r\ge3$ and $t=(r^2+r-2)/2$; this is an integer because $r(r+1)$ is even.
Let $v$ be the vertex of $G=K_r^t$ that carries the pendant vertices.
Then
$m=e(G)=\binom r2+t =r^2-1.$
Vertex $v$ has degree $r-1+t=(r^2+3r-4)/2$, the other clique vertices have degree $r-1$, and the pendant vertices have degree $1$.
Hence $\Delta(G)=(r^2+3r-4)/2$.

Let $\lambda=\lambda(G)$, and normalize a Perron vector $\bm{x}$ so that $x_v=1$.
By symmetry, $\bm{x}$ takes a common value $b$ on the other clique vertices and a common value $c$ on the pendant vertices.
The eigenvector equations at $v$, at another clique vertex, and at a pendant vertex are
\[
\lambda=(r-1)b+tc,\qquad
\lambda b=1+(r-2)b,\qquad
\lambda c=1.
\]
The second equation gives $(\lambda-r+2)b=1$, so $\lambda>r-2$ and $b=1/(\lambda-r+2)$.
The third gives $c=1/\lambda$.
Multiplying the first equation by $\lambda$ and subtracting $m=t+r(r-1)/2$, we obtain
\begin{equation}
\lambda^2-m=\frac{(r-1)\lambda}{\lambda-r+2}-\frac{r(r-1)}{2}
=\frac{(r-1)(r-2)(r-\lambda)}{2(\lambda-r+2)}.
\label{e-2}
\end{equation}
If $\lambda\ge r$, then the right side of~\eqref{e-2} is nonpositive, so $\lambda\le\sqrt m<r$, a contradiction.
Hence $\lambda<r$, the right side of~\eqref{e-2} is positive, and $\sqrt m<\lambda<r$.
Thus $G$ is a Nosal graph.

For the two inequalities, recall that $f(s)=(s-1)(s-2)/2$, and note that $f(r)-f(s)=(r-s)(r+s-3)/2$.
A direct computation gives $m-\Delta(G)=(r^2-3r+2)/2=f(r)$ and $m-f(\sqrt m)=\frac m2+\frac32\sqrt m-1$.
Hence the two quantities in Proposition~\ref{prop-1.3} are $f(r)-f(\sqrt m)$ and $f(r)-f(\lambda)$.
Since $r-\sqrt m=(r^2-m)/(r+\sqrt m)=1/(r+\sqrt m)$, we have
\[
f(r)-f(\sqrt m)=\frac{(r-\sqrt m)(r+\sqrt m-3)}{2}=\frac12-\frac{3}{2(r+\sqrt m)},
\]
which lies in $(0,\frac12)$ because $r+\sqrt m>3$.

For the second quantity, let $d=r-\lambda$, so that $0<d< r- \sqrt{m}= 
1/ (r+ \sqrt{m}) < \frac{1}{2}$.
Since $m=r^2-1$, we have $\lambda^2-m=1-d(r+\lambda)$, and~\eqref{e-2} becomes
\[
1=d\Bigl(r+\lambda+\frac{(r-1)(r-2)}{2(2-d)}\Bigr).
\]
Using $\lambda>r-\frac12$ and $2-d<2$, we obtain
\[
1>d \Bigl(2r-\frac12+\frac{(r-1)(r-2)}{4}\Bigr)=\frac{d r(r+5)}{4}.
\]
Since $\lambda<r$, we conclude that
$0<f(r)-f(\lambda)=\frac{1}{2}d(r+\lambda-3) <d r<\frac{4}{r+5}$.
\end{proof}

\section*{Acknowledgments}
This work was supported by the Postdoctoral Science Foundation of China (No.~2026M793384), the Postdoctoral Fellowship Program of China Postdoctoral Science Foundation (No.~GZC20261656) and the National Natural Science Foundation of China (No.~12601675).
The author would like to thank Hongzhang Chen and Yongjiang Wu for their valuable assistance and suggestions.
During the preparation of this paper, the author used AI tools to brainstorm proof strategies and to provide simplifications of several proof arguments.
All mathematical arguments were reviewed and verified by the author, who takes full responsibility for the content of this paper.

\end{document}